\documentclass[12pt]{article}
\usepackage{amssymb,amsfonts,amsmath,amsthm}
\usepackage{pdfsync}
\newcommand{\one}[1]{\mbox {\bf 1}_{\{#1\}}}

\newcommand{\witi}{\widetilde}
\newcommand{\fracd}[2]{\frac {\displaystyle #1}{\displaystyle #2 }}

\newcommand{\nn}{{\mathbb N}}

\newcommand{\rr}{{\mathbb R}}

\newcommand{\zz}{{\mathbb Z}}
\newcommand{\calf}{{\mathcal F}}

\newcommand{\call}{{\mathcal L}}

\newcommand{\calr}{{\mathcal R}}

\newcommand{\caln}{{\mathcal N}}
\newcommand{\calx}{{\mathcal X}}
\newcommand{\calz}{{\mathcal Z}}
\newcommand{\caly}{{\mathcal Y}}
\newcommand{\veps}{\varepsilon}

\newcommand{\beq}{\begin{eqnarray*}}
\newcommand{\feq}{\end{eqnarray*}}
\newcommand{\beqn}{\begin{eqnarray}}
\newcommand{\feqn}{\end{eqnarray}}
\newcommand{\as}{\mbox{\rm a.~s.}}
\newcommand{\io}{\mbox{\rm i.~o.}}
\newtheorem{theorem}{Theorem}
\makeatletter \@addtoreset{theorem}{section}\makeatother

\newcounter{vadik}

\newtheorem{lemma}[theorem]{Lemma}
\newtheorem{assume}[theorem]{Assumption}
\newtheorem*{theorem*}{Theorem}
\newtheorem*{lemmaa*}{Lemma~A}
\newtheorem{proposition}[theorem]{Proposition}
\newtheorem{corollary}[theorem]{Corollary}
\newtheorem{remark}[theorem]{Remark}

\title{On random coefficient INAR(1) processes}
\author{Alexander Roitershtein\footnote{Department of Mathematics, Iowa State University, Ames, IA 50011; e-mail: roiterst@iastate.edu}
\and Zheng Zhong\footnote{Department of Mathematics, Iowa State University
Ames, IA 50011; e-mail: zz109@iastate.edu}}
\begin{document}
\maketitle
\begin{abstract}
The random coefficient integer-valued autoregressive process was
introduced by Zheng, Basawa, and Datta in \cite{rcinar4}. In this paper we study
the asymptotic behavior of this model (in particular, weak limits of extreme values
and the growth rate of partial sums) in the case where the additive term in the underlying
random linear recursion belongs to the domain of attraction of a stable law.
\end{abstract}
{\em MSC2000: } 62M10, 60J80, 60K37, 60F05.\\
\noindent
{\em Keywords:} Models for count data, thinning models, INAR models,
branching processes, random environment, asymptotic distributions, limit theorems.
\section{Introduction}
\label{intro}
In this paper we consider a {\em first-order random coefficient integer-valued autoregressive} (abbreviated as RCINAR(1))
process that was introduced by Zheng, Basawa, and Datta in \cite{rcinar4}. While the article \cite{rcinar4} as well as a subsequent
work have been focused mostly on direct statistical applications of the model,
the primary goal of this paper is to contribute to the understanding of its probabilistic structure.
\par
Let $\Phi:=(\phi_n)_{n\in\zz}$ be an i.i.d. sequence of reals, each one taking values in the closed interval $[0,1].$
Further, let $\calz:=(Z_n)_{n\in\zz}$ be a sequence of i.i.d. integer-valued
non-negative random variables, independent of $\Phi.$ The pair $(\Phi,\calz)$ is referred to in \cite{rcinar4} as
a sequence of random coefficients associated with the model. 
\par
Let $\zz_+$ denote the set of non-negative integers
$\{n\in\zz:n\geq 0\}.$ The RCINAR(1) process $\calx:=(X_n)_{n\in\zz_+}$ is then defined as follows.
Let $B:=(B_{n,k})_{n\in\zz,k\in\zz}$ be a collection of Bernoulli random variables independent of $\calz$ and such that,
given a realization of $\Phi,$ the variables $B_{n,k}$ are independent and
\beq
P_\Phi(B_{n,k}=1)=\phi_n \qquad \mbox{and} \qquad P_\Phi(B_{n,k}=0)=1-\phi_n,\qquad \forall~k\in \nn,
\feq
where $P_\Phi$ stands for the underlying probability measure conditional on $\Phi.$ Let $X_0=0$ and consider the following linear recursion:
\beqn
\label{version}
X_n=\sum_{k=1}^{X_{n-1}} B_{n,k}+Z_n,\qquad n\in \nn,
\feqn
where we make the usual convention that an empty sum is equal to zero. To emphasize the formal dependence on
the initial condition, we will denote the underlying probability measure (i.e., the joint law of $\Phi,\calz,B,$ and $\calx$)
conditional on $\{X_0=0\}$ by $P_0$ and denote the corresponding expectation by $E_0.$ For the most of the paper 
we will consider a natural initial assumption $X_0=0$ and hence consistently state our results for the measure $P_0.$ 
We remark however that all our results (stated below in Section~\ref{results}) are robust with respect to the initial condition $X_0.$
\par
The RCINAR(1) process $\calx$ defined by \eqref{version}
is a generalization of the {\em integer-valued autoregressive of order one} (abbreviated as INAR(1)) model,
in which the parameters $\phi_n$ are deterministic and identical for all $n\in\zz.$
The model introduced in \cite{rcinar4} has been further extended in \cite{rcinar7,rcinar3,rcinar5,rcinar6,rcinar10,rcinar,rcinar4}.
We refer the reader to \cite{asurvey,esurvey,bsurvey,survey} for a general review of integer-valued (data counting)
time series models and their applications.
\par
Formally, RCINAR(1) can be classified as a special kind of branching processes with immigration in the random environment
$\Phi,$ cf. \cite{mbpire}. In particular, the process can be rigorously constructed on the state space of ``genealogical trees''
(see \cite[Chapter~\setcounter{vadik}{6}$\mbox{\Roman{vadik}}$]{br-book}).
The random variable $X_n$ is then interpreted as the total number of individuals present at generation $n.$
At the beginning of the $n$-th period of time, $Z_n$ immigrants enter the system. Simultaneously and independently of it, each particle from the previous generation
exits the system, producing in the next generation either one child (with probability $\phi_n$) or none
(with the complementary probability  $1-\phi_n$)
\label{note}\footnote{Alternatively, one can think that each particle either survives to the next generation (with probability $\phi_n$)
or dies out (with probability $1-\phi_n$).}.
The branching processes interpretation is a useful point of view on RCINAR(1)
which provides powerful tools for the asymptotic analysis of the model.
\par
Let $\caln_+$ denote the set of non-negative integer-valued random variables in the underlying probability space.
The first term on the right-hand side of \eqref{version} can be thought of as the result of applying to $X_n$ a {\em binomial thinning operator}
which is associated with $\phi_n.$ More precisely, using the following operator notation introduced by Steutel and van~Harn in \cite{operator}:
\beq
\phi_n \circ X:=
\sum_{n=1}^X B_{n,k},\qquad X\in\caln_+,
\feq
equation \eqref{version} can be written as
\beqn
\label{ost}
X_n=\phi_n\circ X_{n-1}+Z_n,\qquad n\in\nn.
\feqn
This form of the recursion indicates that an insight into the probabilistic structure of the RCINAR(1) process
can be gained by comparing it to the classical AR(1) {(\em first-order autoregressive)} model for real-valued data.
The latter is defined by means of i.i.d. pairs $(\phi_n,Z_n)_{n\in\zz}$ of real-valued random coefficients, through
the following linear recursion:
\beqn
\label{ar1}
Y_n=\phi_n Y_{n-1}+Z_n, \qquad n \in \nn.
\feqn
In this paper we explore one of the aspects of the similarity between the RCINAR(1)
and AR(1) processes. Namely, we show in Theorem~\ref{main} below that if $Z_n$ are in the domain of attraction of a stable law
so is the limiting distribution of $X_n,$ and then consider some implications
of this result for the asymptotic behavior of the sequence $X_n.$
A prototype of our Theorem~\ref{main} for AR(1) processes has been obtained in \cite{grey,trakai75}.
Our proof of Theorem~\ref{main} relies on an adaptation of the technique which has been developed
in \cite{grey}.
\par
We conclude the introduction with the following remarks on the motivation for our study. 
Although it appears that most of our results (stated in Section~\ref{results} below) could be extended to a more general type of
processes than is considered here, we prefer to focus on one important model. 
It is well-known that certain \emph{quenched} characteristics of branching processes in random environment satisfy the linear 
difference equation \eqref{ar1}. In two different settings, both yielding stationary solutions 
to \eqref{ar1} with regularly varying tails, this observation has been used to obtain the asymptotic behavior 
of the extinction probabilities in a branching processes in random environment \cite{grey,bpre} 
and the cumulative population for branching processes in random environment with immigration \cite{kks1,mbpire}. 
These studies make it appealing to consider a model like \eqref{ost} which evidently 
combines features of both branching processes in random environment (with immigration) and AR(1) time series.      
\par
In general, probabilistic analysis of the future behavior of average and extreme value characteristics of the underlying system might be handy for typical real-world applications of a counting data model. Our results thus constitute a natural complement to the statistical inference tools developed for the RCINAR(1) processes in \cite{rcinar4}. For the sake of example, consider 
\begin{itemize}
\item [1)] maximal number of unemployed per month in an economy,
according to the model discussed in \cite[Section~1]{rcinar4}; 
\item [2)]
a variation of the model for city size distributions 
studied in \cite{city,city1} where the underlying AR(1) equation is replaced by its suitable integer-valued 
analogue. More precisely, while it is argued in \cite{city,city1} that the evolution of the normalized (to the total size of the population) 
size of a city $Y_n$ obeys \eqref{ar1}, we propose \eqref{ost} as a possible alternative model for non-normalized size of
the city population $X_n,$ where $\phi_n$ is an average proportion of the population which will continue to live 
in the city in the observation epoch $n+1$ and $Z_n$ is the factor accumulating both the natural population growth and migration; 
\item [3)] total number of arrivals in the random coefficient variation of
the queueing system proposed in \cite[Section~3.2]{applic4}. 
\end{itemize}
On the technical side, in contrary to \cite{rcinar4}, we do not restrict ourselves to a setup with $E[Z_0^2]<\infty.$ 
This finite variance condition apparently does not pose a real limitation on the possibility of 
applications of RCINAR(1) to, say, the unemployment rate and the cities growth models mentioned above. 
In both the cases, it is reasonable to assume that the innovations $Z_n$ are typically relatively small
comparing to $X_n$ and, furthermore, large fluctuations of their values are not very likely to occur. 
However, the situation seems to be quite different if one wishes to apply the theory of RCINAR(1) processes 
to the models of queueing theory (as it has been done in \cite{applic4}) when the latter are assumed to 
operate under a heavy traffic regime. See, for instance, 
\cite{qinput,qinput1,qinput3,qinput4,qinput5,qinput7} and \cite{input,input3,input4} for queueing network models 
where it is assumed that the network input has sub-exponential or, more specifically,
regularly varying distribution tails (typically resulting from the distribution of the length of ON/OFF periods). 
We remark that the extensive literature on queueing networks in a heavy traffic regime is
partially motivated by the research on the Internet network activity where it has been shown that in many instances a web traffic 
is well-described by heavy-tailed random patterns; see, for instance, \cite{inet,inet1,inet3,inet4}. 
\section{Statement of results}
\label{results}
This section contains the statement of our main results, and is structured as follows.
We start with a formulation of our specific assumptions on the coefficients $(\Phi,\calz)$ of the model
(see Assumptions~\ref{assume1} and~\ref{assume4} below). Proposition~\ref{convergence} then
ensures the existence of the limiting distribution of $X_n$ and also states formally some related 
basic properties of this Markov chain. Theorem~\ref{main} is concerned with the asymptotic of the tail of the limiting distribution in the case where
the additive coefficients $Z_n$ belong to the domain of attraction of a stable law.
The theorem shows that in this case, the tails of the limiting distribution inherit the structure of the tails of $Z_0.$
This observation leads us to Theorem~\ref{extremes}, which is an extreme value limit theorem
for the sequence $(X_n)_{n\in\zz_+}.$ Weak convergence of suitably normalized partial sums of $X_n$ is the content of
Theorems~\ref{partial} and~\ref{partial4}. The proof of these limit theorems exploit the branching process
representation of a regenerative structure which is described by Proposition~\ref{rect}. Two curious implications of the
existence of this regenerative structure are stated in Propositions~\ref{aage} and~\ref{coalescence}.
The proofs of main theorems stated below in this section are given in Section~\ref{proofs} while the proofs of two auxiliary propositions are deferred
to the Appendix.
\paragraph{Specific assumptions on the random coefficients.}
Recall that a function $f:\rr\to\rr$ is called regularly varying if $f(t)=t^\alpha L(t)$ for some $\alpha\in\rr$ and a function $L$ such that
$\lim_{t\to\infty} L(\lambda t)/L(t)=1$ for all $\lambda >0.$ The parameter $\alpha$ is called the index of the regular variation.
If $\alpha=0,$ then $f$ is said to be slowly varying. We will
denote by $\calr_\alpha$ the class of all regularly varying real-valued functions with index $\alpha.$
We will impose the following assumption on the coefficients of the model defined by \eqref{version}.
\begin{assume}
\label{assume1}
\item [(A1)] $P(\phi_0=1)<1.$
\item [(A2)] For some $\alpha>0,$ there exists $\,h\in \calr_\alpha$ such that $\lim_{t\to\infty} h(t) \cdot P(Z_n>t)=1.$
\end{assume}
Throughout the paper we will assume (actually, without loss of generality in view of \emph{(A2)} and Theorem~1.5.4 in \cite{rvbook}
which ensures the existence of a non-decreasing equivalent for $h$)
that the the following condition is included in Assumption~\ref{assume1}:
\begin{flushleft}
\emph{(A3)} Let $h:(0,\infty)\to\rr$ be as in \emph{(A2)}. Then $\sup_{t>0} h(t) \cdot P(Z_n>t)<\infty.$
\end{flushleft}
\par
The assumption of heavy-tailed innovations (noise terms) in autoregressive models is quite common 
in the applied probability literature. It is a well-known paradigm that such an assumption yields a rich probabilistic structure 
of the stationary solution and allows for a great flexibility in the modeling of its asymptotic behavior. 
See for instance \cite{grey,trakai75}, more recent articles \cite{ht1,ht7,mrec5,ht5,mariana,ht3,ht4}, 
and references therein.    
\par
In a few occasions (including a central limit theorem stated below in Theorem~\ref{clt}) we will use
the following weaker version of Assumption~\ref{assume1}:
\begin{assume}
\label{assume4}
Condition (A1) of Assumption~\ref{assume1} is satisfied and, in addition, the following holds:
\item [(A4)] $E[Z_0^\beta]<\infty$ for some $\beta>0.$
\end{assume}
Assumption~\ref{assume4} is stronger than the usual $E(\log^+ |Z_0|) <+\infty,$ where $x^+:=\max\{x,0\}$ for $x\in\rr,$ 
which is essentially required for the existence and uniqueness of the stationary solution to \eqref{ost}. 
It can be seen through the formula $E[Z_0^\beta]=\int_0^\infty \beta x^{\beta-1}P(Z_0>x)dx$ (recall that $Z_0\geq 0$) that (A4) is basically equivalent 
to the assumption that the distribution tails of $Z_0$ are ``not too thick".  
\paragraph{Limiting distribution of $X_n.$}
Let $Y_n \Rightarrow Y_\infty$ stand for the convergence in distribution of a sequence of random variables $(Y_n)_{n\in\nn}$ to a random variable
$Y_\infty$ (we will usually omit the indication ``as $n\to\infty$"). We will use the notation $X=_D Y$ to indicate that
the distributions of random variables $X$ and $Y$ coincide under the law $P_0.$ For $X\in\caln_+$ define $\Pi_0\circ X:=X$
and, recursively, $\Pi_{k+1}\circ X:=\phi_{k+1}\circ \bigl(\Pi_k\circ X\bigr).$
This defines a sequence of random operators acting in $\caln_+$ as follows:
\beqn
\label{pik}
\Pi_k\circ X=\phi_k\circ \phi_{k-1}\circ\cdots \circ\phi_1 \circ X, \qquad X\in \caln_+.
\feqn
The existence of the stationary distribution for the sequence $\calx=(X_n)_{n\geq0}$ introduced
in \eqref{version} is the content of the following proposition.
\begin{proposition}
\label{convergence}
Let Assumption~\ref{assume4} hold. Then,
\item [(a)] The following series converges to a finite limit with probability one:
\beqn
\label{xlimit}
X_\infty:=\sum_{k=0}^\infty X_{0,k},
\feqn
where the random variables $(X_{0,k})_{k\in\zz_+}$ are independent, and $X_{0,k}=_D\Pi_k \circ Z_0$
for any $k\in\nn.$
\item [(b)] $X_n\Rightarrow X_\infty$ for any $X_0\in\caln_+.$ Here $(X_n)_{n\in\zz_+}$ is understood as the sequence
produced by the recursion rule \eqref{version} with an arbitrary initial value $X_0.$
\item [(c)] The distribution of $X_\infty$ is the unique distribution of $X_0$ which makes $(X_n)_{n\in\zz_+}$
into a stationary sequence.
\end{proposition}
The proof of the proposition is deferred to the appendix. We remark that if $E[Z_0^2]<\infty$ is assumed, the above
statement is essentially Proposition~2.2 in \cite{rcinar4}. For a counterpart of this result for AR(1) processes see, for instance,
Theorem~1 in \cite{stationar1}. It is not hard to deduce from the above proposition the following corollary, whose proof is omitted:
\begin{corollary}
\label{corol}
Suppose that Assumption~\ref{assume4} holds, and let $\calx=(X_n)_{n\in\zz_+}$ be a random sequence defined by \eqref{version}.
Then $\calx$ is an irreducible, aperiodic, and positive-recurrent Markov chain whose stationary measure is supported on
a set of integers $\{k\in\zz_+: k\geq k_{\min}\},$ where $k_{\min}:=\min\{k\in\zz_+:P(Z_0=k)>0\}.$ In particular, $\calx$ is an ergodic sequence.
\end{corollary}
It follows from the above proposition that $X_\infty$ is the unique solution to the distributional fixed point equation $X=_D \phi_0 \circ X +Z_0$
which is independent of $(\phi_0,B_0,Z_0),$ where $B_0$ denotes the sequence $(B_{0,k})_{k\in\nn}.$
In fact, the explicit form \eqref{xlimit} of the stationary distribution along with the identity
$(\phi_n,Z_n)_{n\in\zz}=_D(\phi_{-n},Z_{-n})_{n\in\zz},$ implies that the unique stationary solution to \eqref{version} is given by the following infinite series:
\beqn
\label{stationary}
X_n=\sum_{k=-\infty}^n X_{k,n},
\feqn
where the random variables $(X_{k,n})_{k\in\zz}$ are independent, and
\beq
X_{k,n}=_P \phi_{n-1}\circ \phi_{n-2}\circ\cdots \circ\phi_{k+1} \circ Z_k, \qquad k\leq n.
\feq
By means of the branching process interpretation,
\beqn
\label{xkn}
X_{k,n}=\#\{\mbox{progeny alive at time $n$ of all the immigrants who arrived at time $k$}\},
\feqn
with the convention that $X_{n,n}=Z_n$ and $X_{k,n}=0$ for $k>n.$ Thus \eqref{stationary} states that the stationary solution
to \eqref{version} is formally obtained by letting the zero generation to be formed as a union of the following two groups of individuals:
\begin{enumerate}
\item $Z_0$ immigrants arriving at time zero, and
\item descendants, present in the population at time zero, of all ``demo-immigrants" who has entered the system at the negative times $k=-1,-2,\ldots$
\end{enumerate}
The random variables $X_{k,n}$ can be defined rigorously on the natural state space of the branching process, which is a space of family trees describing
the ``genealogy'' of the individuals (see \cite[Chapter~\setcounter{vadik}{6}$\mbox{\Roman{vadik}}$]{br-book}).
To distinguish between the branching process starting at time zero with $X_0=0$ and its stationary version ``starting at time $-\infty$",
we will denote by $P$ the distribution of the latter, while continuing to use $P_0$ for the probability law of the former. We will denote by $E$ the expectation
operator associated with the probability measure $P.$ We will use the notation $X=_P Y$ to indicate that the distributions of
random variables $X$ and $Y$ coincide under the stationary law $P.$ As it has been mentioned earlier, we will consistently
state our results for the underlying process under the law $P_0$ and thus will consider measure $P$ as an auxiliary tool rather than 
a primary object of interest.   
\par
In the case when the additive term in the underlying random linear recursion belongs to the domain of attraction of a stable law
we have the following
\begin{theorem}
\label{main}
Let Assumption~\ref{assume1} hold. Then,
\beq
\lim_{t\to\infty} h(t)\cdot P(X_\infty>t)=\bigl(1-E[\phi_0^\alpha]\bigr)^{-1}\in (0,\infty).
\feq
\end{theorem}
A prototype of this result for AR(1) processes has been obtained in \cite{grey,trakai75}.
The proof of Theorem~\ref{main} given in Section~\ref{proof-main} relies on an adaptation to our setup of a technique which has been developed
in \cite{grey}.
\paragraph{Extreme values of $\calx.$} We next show that the running maximum of the sequence $\calx$
exhibits the same asymptotic behavior as that of $\calz=(Z_n)_{n\in\zz_+}.$ Let
\beqn
\label{mn}
M_n=\max\{X_1,\ldots,X_n\},\qquad n\in\nn,
\feqn
and
\beqn
\label{bn}
b_n=\inf\{t>0: h(t)\geq n\},
\feqn
where $h(t)$ is the function introduced in Assumption~\ref{assume1}. 
\par
The proof of the following theorem is given in Section~\ref{proof-extremes} below.
\begin{theorem}
\label{extremes}
Let Assumption~\ref{assume1} hold. Then, under the law $P_0,$
\beq
M_n/b_n \Rightarrow M_\infty,
\feq
where $M_\infty$ is a proper random variable with the following distribution function:
\beq
P_0(M_\infty>x)=e^{-x^{-1/\alpha}},\qquad x>0,
\feq
where $\alpha>0$ is the constant introduced in Assumption~\ref{assume1}.
\end{theorem}
The distribution of $M_\infty$ belongs to the class of the so called Fr\'{e}chet extreme value distributions
and in fact (see, for instance, \cite[Section~3.3]{extreM}),
\beq
P_0(M_\infty>x)=\lim_{n\to\infty} P\bigl(\max_{1\leq k\leq n} Z_k>xb_n\bigr),\qquad x>0.
\feq
It is quite remarkable that the distribution of $\phi_0$ does not
play any role in the result of Theorem~\ref{extremes}. An intuitive explanation for this phenomenon, which can be derived from the proof, is as
follows. Due to the basic property of regular variation, two independent terms $\phi_n\circ X_{n-1}$ and $Z_n$ are unlikely
to ``help" each other in creating a large value of the sum $X_{n+1}=\phi_n\circ X_{n-1}+Z_n.$ Moreover, the law of large numbers
ensures that the ratio $\phi_n\circ X_{n-1}/X_{n-1}$ is bounded away from one with an overwhelming probability
whenever $\phi_n\circ X_{n-1}$ is large. Therefore, the asymptotic
of the extreme value of the sequence $X_n$ follows that of $Z_n.$
\par
\paragraph{Regenerative structure of $\calx.$} Let
\beqn
\label{nu}
\nu_0=1~~~\mbox{and}~~~\nu_n=\inf\{i>\nu_{n-1}:\phi_i \circ X_{i-1}=0\},
\feqn
with the usual convention that the infimum over an empty set is $\infty.$
We will refer to $\nu_n$ as a {\em regeneration time} and to the time
elapsing from $\nu_{n-1}$ until $\nu_n-1$ as {\em the $n$-th renewal epoch}.
In the language of branching processes, at the regeneration times the extinction occurs and
and the process starts again with the next wave of the immigration.  For $n\in\nn,$ let
\beq
\sigma_n=\nu_n-\nu_{n-1}\qquad \mbox{\rm and}\qquad R_n=\bigl(X_i:\nu_{n-1}\leq i< \nu_n\bigr)
\feq
be, respectively, the length of the $n$-th renewal epoch and the list of the values of $X_i$
recorded during the $n$-th renewal epoch.
\par
The proof of the following proposition is given in the appendix.
\begin{proposition}
\label{rect}
Let Assumption~\ref{assume4} hold. Then,
\item [(a)] $P_0(\nu_n<\infty)=1$ for all $n\in\nn.$ Moreover, the pairs $(\sigma_n,R_n)_{n\in\nn}$ form an i.i.d. sequence.
\item [(b)] There exist positive constants $K_1>0$ and $K_2>0$ such that
\beqn
\label{estimate}
P_0(\sigma_1>t)\leq K_1e^{-K_2t},\qquad \forall~t>0.
\feqn
\end{proposition}
While the first part of the proposition is a standard Markov chain exercise,
the exponential bound in \eqref{estimate} is a delicate result. A similar bound has been proved
for a general type of branching processes with immigration in \cite{mbpire}. An argument which is due to M.~Kozlov
and which has been adapted for the proof of Theorem~4.2 in \cite{mbpire} goes through
almost verbatim for our setting. We provide a suitable variation of this argument in the appendix.
\par
The existence of the ``life-cycles" (i.e., renewal epochs) for the branching process implies, for instance,
the following. Recall $X_{k,n}$ from \eqref{xkn}. Let
\beq
\lambda_n=n-\max\{k<n: X_{k,n}> 0\}
\qquad
\mbox{and}
\qquad
\eta_n=\frac{\sum_{k=1}^n X_{k,n}\cdot(n-k)}{\sum_{k=1}^n X_{k,n}}
\feq
be, respectively, the maximal and the average age of the individuals present at generation $n$ (see the above footnote remark on p.~\pageref{note}).
\begin{proposition}
\label{aage}
Let Assumption~\ref{assume4} hold. Then both $\lambda_n$ and $\eta_n$ converge weakly, as $n\to\infty,$
to proper distributions. More precisely, under the law $P_0,$
\beq
\lambda_n\Rightarrow \sigma_1 \cdot U
\qquad \mbox{\rm and} \qquad
\eta_n\Rightarrow -\frac{\sum_{k=-\infty}^0 X_{k,0}\cdot k}{\sum_{k=-\infty}^0 X_{k,0}},
\feq
where $U$ is a random variable which is independent of $\sigma_1$ and is distributed uniformly over the interval $[0,1].$
\end{proposition}
The first result in the above proposition is a direct implication of the renewal theorem
whereas the second one is a consequence of the explicit formula for $\eta_n$ given above and the fact that
$X_{k,0}=0$ for $k<\nu_{-1}$ and $P(\nu_{-1}<\infty)=1.$ Here $\nu_{-1}$ is time of the last renewal up to
time zero for the process starting at $-\infty.$ We leave details to the reader.
\par
Another interesting implication of the existence of the regenerative structure is the convergence
in distribution of the coalescence time at generation $n.$ Suppose that $X_n>2$ and sample at random
two individuals present at generation $n.$ Then the coalescence time $T_n$ is defined as
$n-k$ if the immigrant ancestors of both individuals have entered the system at the same time $k\in\zz_+,$
and is set to be infinity otherwise (cf., for instance, \cite{coalescence}). Since the probability of sampling
of both individuals among the descendants of the immigration wave $k$ is $\frac{X_{k,n}(X_{k,n}-1)}{X_n(X_n-1)},$
\beq
P_0(T_n\leq t)=E\left[\frac{\sum_{k=n-t}^n X_{k,n}(X_{k,n}-1)}{\sum_{k=1}^n X_{k,n}\Bigl(\sum_{k=1}^n X_{k,n}-1\Bigr)} \right].
\feq
We have thus obtained the following:
\begin{proposition}
\label{coalescence}
Let Assumption~\ref{assume4} hold. Then $T_n$ converges weakly  under $P_0,$ as $n\to \infty,$
to a proper random variable with the following distribution function on $\nn \bigcup \{0,+\infty\}:$
\beq
F(t)=E\left[\frac{\sum_{k=-t}^0 X_{k,0}(X_{k,0}-1)}{\sum_{k=-\infty}^0 X_{k,0}\Bigl(\sum_{k=-\infty}^0 X_{k,0}-1\Bigr)} \right], \qquad t<\infty,
\feq
where $\frac{0}{0}$ inside the expectation is interpreted as $0.$
\end{proposition}
\paragraph{Growth rate and fluctuations of the partial sums of $\calx.$}
Let $S_n=\sum_{k=1}^n X_k.$ The following law of large numbers is a direct consequence of Corollary~\ref{corol}.
\begin{proposition}
\label{lln}
Let Assumption~\ref{assume4} hold with $\beta=1.$  Then
\beq
\lim_{n \to \infty} \fracd{S_n}{n} =E[X_0]=\frac{E[Z_0]}{1-E[\phi_0]},\qquad P_0-\as
\feq
\end{proposition}
The next theorem is concerned with the rate of the growth of
the partial sums when $Z_0$ has infinite mean.
For $\alpha \in (0,2]$ and $b >0$ denote by $\call_{\alpha,b}$
the strictly asymmetric stable law of index $\alpha$ with the characteristic function
\beqn \label{kappa-law} \log \widehat
\call_{\alpha,b}(t)=-b|t|^\alpha\left(1+i\fracd{t}{|t|}
f_\alpha(t)\right), \feqn where $f_\alpha(t)=-\tan
\fracd{\pi}{2}\alpha$ if $\alpha \neq 1,$ $f_1(t)=2/\pi \log t.$
With a slight abuse of notation we use the same symbol for the
distribution function of this law. If $\alpha<1,$
$\call_{\alpha,b}$ is supported on the positive reals, and if
$\alpha \in (1,2],$ it has zero mean \cite[Section~2.2]{extreM}.
\par
Recall $b_n$ from \eqref{bn}. The following result is proved in Section~\ref{partial-proof} below
by using an approximation of the partial sums of the process by those of a
stationary strongly mixing sequence for which we are able to verify the conditions
of a general stable limit theorem.
\begin{theorem}
\label{partial} Let Assumption \ref{assume1} hold with $\alpha\in (0,1).$
Then $b_n^{-1}S_n \Rightarrow \call_{\alpha,b}.$
\end{theorem}
We next study the fluctuations of the partial sums in the case where non-trivial centering of $X_n$ is required
to obtain a proper weak limit for the partial sums.
\begin{theorem}
\label{partial4} Let Assumption \ref{assume1} hold with $\alpha\in [1,2].$
For $n\in\nn,$ define
\beq
a_n=
\left\{
\begin{array}{lr}
b_n, ~\mbox{\rm where $b_n$ is defined in \eqref{bn}},&~\mbox{\rm if}~\alpha<2,\\
\inf\,\{t>0: nt^{-2}\cdot E\bigl[X_0^2;\,X_0\leq t\bigr]\leq 1\}~&\mbox{\rm if}~\alpha=2.
\end{array}
\right.
\feq
Denote $\mu:=E[X_0].$ Then the following holds for some $b>0:$
\item[(i)] If $\alpha=1,$ then $a_n^{-1}(S_n-c_n) \Rightarrow \call_{1,b}$ with $c_n=nE[X_0;X_0\leq a_n].$
\item[(ii)] If $\alpha \in (1,2),$ then $a_n^{-1}(S_n-n \mu) \Rightarrow \call_{\alpha,b}.$
\item[(iii)] If $\alpha=2$ and $E[Z_0^2]=\infty,$ then $a_n^{-1}(S_n-n\mu)\Rightarrow \call_{2,b}.$
\end{theorem}
Recall $\nu_n$ from \eqref{nu} and define
\beq
W_n=\sum_{i=\nu_{n-1}}^{\nu_n-1}X_i,\qquad n\in\nn .
\feq
Theorem~\ref{partial4} can be derived from stable limit theorems for
partial sums of i.i.d. variables, using the regenerative structure and the following lemma.
\begin{lemma}
\label{tails}
Let Assumption~\ref{assume1} hold. Then the following limit exists:
\beq
\lim_{t\to\infty} h(t)\cdot P_0(W_1>t).
\feq
Moreover, the limit is finite and strictly  positive.
\end{lemma}
The proof of the lemma given below in Section~\ref{tails-proof} is (although technical details are quite different) along the line of the proof of a similar result given
for a different branching process in \cite{kks1}. We remark that though a similar technique can be used to obtain Theorem~~\ref{partial},
we prefer to employe a more direct approach in the case $\alpha\in(0,1).$
Theorem~\ref{partial4} follows from the above lemma by using a standard argument, which is outlined in \cite{kks1} in for the case $h(x)=x^{-1}.$
Since only an obvious minor modification is required to extend the argument to a general $h$ (see, for instance, \cite{stable4} for $h(x)=x^{-2}$),
we omit details of this argument here.
\par
If an appropriate second moment condition is assumed, one can establish the following functional
limit theorem for normalized partial sums of $\calx.$ Let $D(\rr_+,\rr)$ denote the set of real-valued c\`{a}dl\`{a}g functions
on $\rr_+:=[0,\infty),$ endowed with the Skorokhod $J_1$-topology. Let $\lfloor x \rfloor$ denote the integer part of $x\in\rr.$ We have:
\begin{theorem}
\label{clt}
Let Assumption~\ref{assume4} hold with a constant $\beta>2.$
Then, as $n\to\infty$, the sequence of processes
\beq
S_t^{(n)}=n^{-1/2}\bigl(S_{\lfloor nt \rfloor}-nt\mu \bigr),\qquad t\in[0,1].
\feq
in $D(\rr_+,\rr)$ converges weakly to a non-degenerate Brownian motion $W_t,$ $t\in [0,1].$
\end{theorem}
Theorem~\ref{clt} is a particular case of \cite[Theorem~1.5]{multi}, and therefore its proof is omitted.
Notice that the conditions of the theorem are satisfied if Assumption~\ref{assume1} holds with $\alpha>2.$
\section{Proof of the main results}
\label{proofs}
This section is devoted to the proof of the theorems stated in Section~\ref{results}
(namely, Theorems~\ref{main}, \ref{extremes}, ~\ref{partial} and~\ref{partial4}),
and is divided into four subsections correspondingly.
\subsection{Proof of Theorem~\ref{main}}
\label{proof-main}
First, we observe the following.
\begin{lemma}
\label{klemma}
Let $X\in\caln_+$ be a random variable in the underlying probability space such that
\item [(i)] $X$ is independent of $(\phi_n,Z_n,B_n)_{n\in\zz_+},$ where $B_n:=(B_{n,k})_{k\in \nn}.$
\item [(ii)] $\lim_{t\to\infty} h(t) \cdot P_\Phi(X>t)=1$ for some $h\in \calr_\alpha,$ $\alpha>0.$
\\
Then  $\lim_{t\to\infty} h(t) \cdot P_\Phi(\phi_0 \circ X >t)=\phi_0^\alpha.$
\end{lemma}
\begin{proof}[Proof of Lemma~\ref{klemma}]
Fix a constant $\veps\in (0,1).$ For $t>0$ define the following three events:
\beq
A_{t,\veps}&=&\bigl\{X>t\cdot (\phi_0^{-1}+\veps)\bigr \},\\
B_{t,\veps}&=&\bigl\{t\cdot (\phi_0^{-1}-\veps)  < X\leq t\cdot (\phi_0^{-1}+\veps)\bigr \},\\
C_{t,\veps}&=&\bigl\{X\leq t\cdot (\phi_0^{-1}-\veps) \bigr \}.
\feq
We will use the following splitting formula:
\beq
P_\Phi\bigl(\phi_0 \circ X>t\bigl)=P_\Phi\bigl(\phi_0 \circ X>t;\, A_{t,\veps}\bigr)+
P_\Phi\bigl(\phi_0 \circ X>t;\, B_{t,\veps}\bigr)
+P_\Phi\bigl(\phi_0 \circ X>t;\, C_{t,\veps}\bigr).
\feq
By the law of large numbers,
\beq
\lim_{n\to\infty} \frac{1}{n}\sum_{k=1}^n B_{1,k}=\phi_0,\qquad P-\as
\feq
Since $h(t)$ is regularly varying, Chernoff's bound  (Cram\'{e}r's large deviation
theorem for coin flipping, see \cite{ldpbook}) applied to the partial sums $\sum_{k=1}^n B_k$ implies that
\beq
0\leq \limsup_{t\to\infty} h(t)\cdot P_\Phi\bigl(\phi_0 \circ X>t;\, C_{t,\veps}\bigr)
\leq \limsup_{t\to\infty} h(t)\cdot P_\Phi\Bigl({\sum}_{k=1}^{\lfloor t(\phi_0^{-1}-\veps)\rfloor} B_k> t\Bigr)=0.
\feq
Next, by the conditions of the lemma,
\beq
&&
\lim_{t\to\infty} h(t)\cdot P_\Phi\bigl(\phi_0 \circ X>t;\, B_{t,\veps}\bigr)\leq \lim_{t\to\infty} h(t)\cdot P_\Phi(B_{t,\veps})
\\
&&
\qquad
=\bigl[(\phi_0^{-1}-\veps)^{-\alpha}-(\phi_0^{-1}+\veps)^{-\alpha}\bigr]\to_{\veps \to 0} 0.
\feq
Finally, using again the large deviation principle for $\sum_{k=1}^n B_k,$
\beq
&&\liminf_{t\to\infty}\, h(t)\cdot P_\Phi\bigl(\phi_0 \circ X>t;\, A_{t,\veps}\bigr)=
\liminf_{t\to\infty}\, h(t)\cdot \Bigl[P_\Phi\bigl( A_{t,\veps}\bigr)-P_\Phi\bigl(\phi_0 \circ X\leq t;\, A_{t,\veps}\bigr)\Bigr]
\\
&&
\qquad
\geq \liminf_{t\to\infty}\, h(t)\cdot P_\Phi\bigl(A_{t,\veps}\bigr)
=
(\phi_0^{-1}+\veps)^{-\alpha}.
\feq
On the other hand, clearly,
\beq
\liminf_{t\to\infty}\, h(t)\cdot P_\Phi\bigl(\phi_0 \circ X>t;\, A_{t,\veps}\bigr)\leq \liminf_{t\to\infty}\, h(t)\cdot P_\Phi\bigl(A_{t,\veps}\bigr)=
(\phi_0^{-1}+\veps)^{-\alpha}.
\feq
Since $\veps>0$ is arbitrary and $(\phi_0^{-1}+\veps)^{-\alpha}
\to \phi_0^\alpha$ as $\veps$ goes to zero, this completes the proof of the lemma.
\end{proof}
\begin{remark}
\label{remark1}
The above proof of Lemma~\ref{klemma} can be adopted without modification for a more general type of sums $\sum_{k=1}^X B_k,$
where $X\in\caln_+$ has regularly varying distribution tails and $(B_k)_{k\in\nn}$ are independent of $X.$
In fact, the only property of the sequence $B_k$ required by the proof is
the availability of a non-trivial large deviations upper bound for its partial sums.
Note that if $f(\lambda):=E_\Phi\bigl[e^{\lambda B_1}\bigr]$ is finite in a neighborhood of zero,
such a bound in the form $P_\Phi\Bigl(\Bigl|\frac{1}{n}\sum_{k=1}^n B_k-E_\Phi[B_1]\Bigr|>x\Bigr)\leq c(x)e^{-nI(x)}$
with suitable constants $c(x),I(x)>0$ holds for any $x>0$ (see, for instance, the first inequality in the proof of Lemma~2.2.20 in \cite{ldpbook}).
\end{remark}
Recall (see, for instance, \cite[ Lemma~1.3.1]{extreM}) that if $X$ and $Y$ are two independent random variables
such that $\lim_{x\to\infty} h(x)\cdot P(X>x)=c_1>0$ and $\lim_{x\to\infty} h(x)\cdot P(Y>x)=c_2>0$
for some $h\in\calr_\alpha,$ $\alpha>0,$ then
\beqn
\label{sumrv}
\lim_{x\to\infty} h(x)\cdot P(X+Y>x)=c_1+c_2.
\feqn
Using this property and iterating \eqref{version}, one can deduce from Lemma~\ref{klemma} the following corollary.
Consider (in an enlarged probability space, if needed) a sequence $\witi \calx=\bigl(\witi X_n\bigr)_{n\in\zz_+}$ which solves \eqref{version},
that is a sequence such that
\beqn
\label{version4}
\witi X_n=\sum_{k=1}^{\witi X_{n-1}} B_{n,k}+Z_n,\qquad n\in \nn,
\feqn
for some initial (not necessarily equal to zero) random value $\witi X_0.$
\begin{corollary}
\label{icor}
Let Assumption~\ref{assume1} hold and suppose in addition that the following two conditions are satisfied:
\begin{itemize}
\item [(i)] $\witi X_0$ is independent of $(\phi_k,B_k,Z_k)_{k>0},$ where $B_k=(B_{k,j})_{j\in \nn}.$
\item [(ii)] $\lim_{t\to\infty} h(t) \cdot P_\Phi\bigl(\witi X_0>t\bigr)=c_0$ for some random variable $c_0=c_0(\Phi).$
\end{itemize}
Then  $\,\lim_{t\to\infty} h(t) \cdot P_\Phi\bigl(\witi X_n>t\bigr)=c_n$ for any $n \in\nn,$
where the random variables $c_n=c_n(\Phi)$ are defined recursively by
\beqn
\label{recs}
c_{n+1}=c_n \phi_{n+1}^\alpha +1,\qquad n\in\zz_+.
\feqn
\end{corollary}
The recursive relation \eqref{recs} implies that
\beqn
\label{cnbar}
c_n=\bar c_n+c_0\prod_{j=1}^n \phi_j^\alpha ,\qquad \mbox{where}\qquad \bar c_n=1+\sum_{k=2}^n \prod_{j=k}^n \phi_j^\alpha,
\feqn
and hence (see, for instance, Theorem~1 in \cite{stationar1}) the random variables
$c_n$ converge in distribution, as $n\to\infty,$   to
\beq
c_\infty:=1+\sum_{k=0}^\infty \prod_{i=0}^k \phi_{-i}^\alpha .
\feq
Furthermore, we have the following:
\begin{corollary}
\label{cori}
Suppose that the conditions of Corollary~\ref{icor} are satisfied and, in addition,
there exist a positive constant $C>0$ and such that the following holds:
\beq
P\Bigl(\sup_{t>0} \bigl\{h(t) \cdot P_\Phi\bigl(\witi X_0>t\bigr)\bigr\}<C\Bigr)=1.
\feq
Then the following limit exists and the identity holds:
\beq
\lim_{n\to\infty} h(t)\cdot P\bigl( \witi X_n>t\bigr)=E[c_n],\qquad n\in\nn,
\feq
where $c_n$ are random variables defined in \eqref{cnbar}.
\end{corollary}
\begin{proof}[Proof of Corollary~\ref{cori}]
Corollary~\ref{icor} and the bounded convergence theorem imply that
\beqn
\label{enable}
&&
\lim_{t\to\infty} h(t)\cdot P\bigl(\witi X_n>t\bigr)
=\lim_{t\to\infty} h(t)\cdot E\bigl[P_\Phi\bigl(\witi X_n>t\bigr)\bigr]
\nonumber
\\
&&
\qquad
= E\bigl[\lim_{t\to\infty} h(t) \cdot P_\Phi\bigl(\witi X_n>t\bigr)\bigr]=E[c_n].
\feqn
To justify interchanging of the limit with the expectation, observe that
$\witi X_n\leq \witi X_0+\sum_{k=1}^n Z_k$ and hence, by virtue of assumption \emph{(A3)},
the following inequalities hold with probability one for some positive constant $C_1>0:$
\beq
&&
h(t)\cdot P_\Phi\bigl(\witi X_n>t\bigr) \leq h(t)\cdot P_\Phi\bigl(\witi X_0>t/2\bigr)
+h(t)\cdot P_\Phi\Bigl(\sum_{k=1}^n Z_k>t/2\Bigr)
\\
&&
\qquad
\leq h(t)\cdot P_\Phi\bigl(\witi X_0>t/2\bigr)
+n h(t)\cdot P\bigl(Z_0>t/(2n)\bigr)
\\
&&
\qquad
\leq  C\frac{h(t)}{h(t/2)}+C_1n\frac{h(t)}{h(t/(2n))}.
\feq
It follows (see, for instance, \cite[Lemma~1]{grey}) that there exists a constant $C_2>0$ such that
\beq
P\Bigl(\sup_{t>t_0} \bigl\{h(t)\cdot P_\Phi\bigl(\witi X_n>t\bigr)\bigr\}<C_2\Bigr)=1.
\feq
This enables one to apply the bounded convergence theorem in \eqref{enable} and thus completes the proof of the corollary.
\end{proof}
In what follows notations $X\leq_D Y$ and $X \geq_D Y$ for random variables $X$ and $Y$ are used to indicate that
$P(X>t)\leq P(Y>t)$ or, respectively,  $P(X>t)\geq P(Y>t)$ holds for all $t\in\rr.$
In order to exploit Corollary~\ref{cori} in the proof of Theorem~\ref{main}, we need the following:
\begin{lemma}
Suppose that the conditions of Corollary~\ref{cori} are satisfied. Then:
\label{ordering}
\item [(a)] If $\witi X_0\leq_D \witi X_1,$ then $\witi X_n\leq_D \witi X_{n+1}$ for all $n\in\nn.$
\item [(b)] If $\witi X_0\geq_D \witi X_1,$ then $\witi X_n\geq_D \witi X_{n+1}$ for all $n\in\nn.$
\end{lemma}
\begin{proof}[Proof of Lemma~\ref{ordering}] $\mbox{}$ The proof is by induction.
Suppose first that $\witi X_{n-1}\leq_D \witi X_n$ for some $n\in\nn,$ $\witi X_{n-1}$ is independent of
$(\phi_k,B_k,Z_k)_{k> n-1},$ and $\witi X_n$ is independent of $(\phi_k,B_k,Z_k)_{k> n}.$ We will now use the following standard trick to construct an
auxiliary random pair $(V_{n-1},V_n)$ such that
\beqn
\label{coupling}
P(V_{n-1}\leq V_n=1),\qquad V_{n-1}=_P \witi X_{n-1},\qquad \mbox{and}\qquad  V_n=_P \witi X_n.
\feqn
Let $U$ be a uniform random variable on $[0,1],$ independent of the random coefficients sequence $(\Phi,\calz).$ Denote
by $F_n$ and $F_{n-1},$ respectively, the distribution functions of $X_n$ and $X_{n-1}.$
Set $V_n=F_n^{-1}(U)$ and $V_{n-1}=F_{n-1}^{-1}(U),$ where $F^{-1}(y):=\inf\{x\in\rr : F (x) \geq y\},$ $y\in [0,1],$
with the convention that $\inf \emptyset= \infty.$
\par
Let $\witi X_{n+1}=\phi_{n+1}\circ \witi X_n+Z_{n+1}.$ Then $\witi X_{n+1}$ is independent of $(\phi_k,B_k,Z_k)_{k> n+1}.$
Furthermore, since $(V_{n-1},V_n)$ is independent of $(\Phi,\calz),$ we obtain for any $t>0,$
\beqn
\label{joke}
P\bigl(\witi X_{n+1}>t\bigr)&=&P\bigl(\phi_{n+1} \circ \witi X_n+Z_{n+1}>t\bigr)
=P(\phi_{n+1} \circ V_n+Z_{n+1}>t)
\nonumber
\\
&\geq&
P(\phi_{n+1} \circ V_{n-1}+Z_{n+1}>t)=P(\phi_n \circ V_{n-1}+Z_n>t)
\\
\nonumber
&=&
P\bigl(\phi_n \circ \witi X_{n-1}+Z_n>t\bigr)=P\bigl(\witi X_n>t\bigr).
\feqn
This shows that part (a) of the lemma holds true. The same argument, but with $\leq$ replaced by $\geq$ and vice versa
in the base of induction, \eqref{coupling}, and \eqref{joke}, yields part (b).
\end{proof}
We are now in a position to complete the proof of Theorem~\ref{main}.
First, we have:
\begin{lemma}
\label{domin}
There exists a random variable $\witi X_0\geq 0$ satisfying the conditions of Corollary~\ref{cori},
such that $\witi X_1\geq_D \witi X_0.$
\end{lemma}
\begin{proof}[Proof of Lemma~\ref{domin}]
Set $\witi X_0=Z_{-1}.$
\end{proof}
In view of Lemma~\ref{ordering}, this implies that we can find a sequence
$\witi X_n$ that solves \eqref{version} and such that $\witi X_n\leq_D \witi X_\infty,$
while $\witi X_0$ satisfies the conditions of Corollary~\ref{cori}.
Combining this result with the conclusion of the corollary yields:
\beq
\liminf_{t\to \infty}h(t)\cdot P(X_\infty>t) \geq \lim_{t\to\infty}h(t)\cdot P\bigl(\witi X_n>t\bigr)=E[c_n],\qquad n\in\nn.
\feq
Hence
\beqn
\label{upper}
\liminf_{t\to \infty} h(t)\cdot P(X_\infty>t) \geq \lim_{n\to\infty} E[c_n]=\frac{1}{1-E[\phi_0^\alpha]}.
\feqn
On other hand, we have
\begin{lemma}
\label{domin1}
Let Assumption~\ref{assume1} hold. There exists a random variable $\witi X_0\geq 0$ satisfying the conditions of Lemma~\ref{klemma},
and such that $\witi X_1\leq_D \witi X_0.$
\end{lemma}
\begin{proof}[Proof of Lemma~\ref{domin1}]
Given a realization of the sequence $\Phi,$ choose a constant $c_0$ in such a way that
\beq
c_0>\frac{1}{1-E[\phi_0^\alpha]}.
\feq
Let $Y_0=c_0^{1/\alpha} Z_{-1}.$ Then $\lim_{t\to\infty} h(t) \cdot P(Y_0>t)=c_0.$
If we would choose $\witi X_0=Y_0,$ we would have $c_1:=\lim_{t\to\infty} h(t) \cdot P\bigl(\witi X_1>t\bigr) <c_0$ by virtue of \eqref{recs}
and Corollary~\ref{cori}. This would imply that $P\bigl(\witi X_1>t\bigr)<P\bigl(\witi X_0>t\bigr)$ for $t>t_0,$ where $t_0>0$ is a positive constant which depends
on $c_0.$ Consider now (in an enlarged probability space, if needed)
a random variable $\witi X_0$ such that $\witi X_0$ is independent of $(\phi_k,B_k,Z_k)_{k\in\zz}$ and
\beq
P\bigl(\witi X_0>t\bigr)=P(Y_0>t|Y_0>t_0).
\feq
Note that such $\witi X_0$ satisfies the conditions of Corollary~\ref{cori} because
$P_\Phi\bigl(\witi X_0>t\bigr)=P\bigl(\witi X_0>t\bigr)$ with probability one, and for $t>t_0,$
\beq
h(t)\cdot P\bigl(\witi X_0>t\bigr)\leq \frac{1}{P(c_0^\alpha Z_0>t_0)} \cdot \frac{h(t)}{h(t c_0^{-\alpha})}
\Bigl(h(t c_0^{-\alpha}) \cdot P( Z_0>tc_0^{-\alpha})\Bigr),
\feq
and $\sup_{t>0}h(t)/h(t c_0^{-\alpha})<\infty$ (see, for instance, Lemma~1 in \cite{grey}).
Then, for $t>t_0,$
\beq
P\bigl(\phi_1 \circ \witi X_0 + Z_1>t\bigr)&=&P(\phi_1 \circ Y_0 + Z_1>t|Y_0>t_0)
\\
&=&
\frac{P(\phi_1 \circ Y_0 + Z_1>t;Y_0>t_0)}{P(Y_0>t_0)}\leq
\frac{P(\phi_1 \circ Y_0 + Z_1>t)}{P(Y_0>t_0)}
\\
&\leq & \frac{P(Y_0 >t)}{P(Y_0>t_0)}=P(Y_0 >t|Y_0>t_0)=P\bigl(\witi X_0>t\bigr).
\feq
On the other hand, if $t\leq t_0$ then
\beq
P\bigl(\witi X_0>t\bigr)=P\bigl(\witi X_0>t\bigl|\witi X_0>t_0\bigr)=1.
\feq
Thus
\beq
P\bigl(\phi_1 \circ \witi X_0 + Z_1>t\bigr)\leq P\bigl(\witi X_0>t\bigr)
\feq
for all $t>t_0,$ and we can set $\witi X_0$ as the initial value for the recursion.
\end{proof}
Combining this result with Corollary~\ref{cori} yields:
\beq
\limsup_{t\to \infty} h(t)\cdot P(X_\infty>t)\leq \lim_{t\to\infty} h(t)\cdot P_0(X_n>t)=E[c_n],\qquad n\in\nn.
\feq
Hence,
\beq
\limsup_{t\to \infty} h(t)\cdot P(X_\infty>t) \leq \lim_{n\to\infty} E[c_n]=\frac{1}{1-E[\phi_0^\alpha]}.
\feq
The proof of Theorem~\ref{main} is completed in view of \eqref{upper}.
\subsection{Proof of Theorem~\ref{extremes}}
\label{proof-extremes}
For $n\in\nn,$ denote $K_n=\max_{1\leq k\leq n} Z_k.$
It follows from \eqref{version} that $M_n\geq_D K_n.$ To conclude the proof of the theorem, it thus
suffices to show that
\beq
\limsup_{n\to\infty} P_0(M_n> xb_n) \leq \lim_{n\to\infty} P_0(K_n>xb_n)=e^{-x^{-1/\alpha}},\qquad x>0.
\feq
Observe that, under the stationary law $P,$ the branching process (without immigration) originated
by the initial $X_0$ individuals will eventually die out. Therefore, the total number of progeny
of the individuals in the zero generation is $P-\as$ finite. Furthermore, the branching process $X_n-\sum_{k=-\infty}^0 X_{k,n},$
$n\in\nn,$ obtained by excluding the contribution of these individuals from the original one, is distributed
under $P$ as $X_n,$ $n\in\nn,$ under $P_0.$ It thus suffices to show that
\beq
\limsup_{n\to\infty} P(M_n> xb_n) \leq \lim_{n\to\infty} P(K_n>x b_n)=e^{-x^{-1/\alpha}},\qquad x>0.
\feq
Toward this end, define the following events. For $x>0,$ $\delta>0,$ and $\veps\in (0,1/2),$ let
\beq
A_{x,\delta}^{(n)}&=&\{xb_n<M_n\leq x(1+\delta)b_n\},\qquad n\in\nn,
\\
B_{x,\delta,\veps}^{(n)}&=&A_{x,\delta}^{(n)}\bigcap \{x(1-\veps)b_n<K_n\leq x(1+\delta)b_n\},\qquad n\in\nn,
\\
C_{x,\delta,\veps}^{(n,k)}&=&A_{x,\delta}^{(n)}\bigcap \{X_k>xb_n,\,\veps xb_n< Z_k\leq x(1-\veps)b_n\},\qquad n\in\nn,~k=1,\ldots,n,
\\
D_{x,\delta,\veps}^{(n,k)}&=&A_{x,\delta}^{(n)}\bigcap \{X_k>xb_n,\,Z_k\leq x \veps b_n\},\qquad n\in\nn,~k=1,\ldots,n.
\feq
Then
\beqn
\label{cdecomp}
\nonumber
&&
P\bigl(A_{x,\delta}^{(n)}\bigr)
\leq P\bigl(B_{x,\delta,\veps}^{(n)}\bigr)
+P\Bigl(\bigcup_{k=1}^n C_{x,\delta,\veps}^{(n,k)}\Bigr)
+P\Bigl(\bigcup_{k=1}^n D_{x,\delta,\veps}^{(n,k)}\Bigr)
\\
&&
\quad
\leq
P\Bigl(x(1-\veps)b_n<K_n\leq x(1+\delta)b_n\Bigr)
+n\,P\bigl(C_{x,\delta,\veps}^{(n,1)}\bigr)
+n\,P\bigl(D_{x,\delta,\veps}^{(n,1)}\bigr).
\feqn
Taking into account the independence of the pair $(\phi_k,X_{k-1})$ of $Z_k,$
it follows from \eqref{ost}, Assumption~\ref{assume1}, and Lemma~\ref{klemma} that
for any positive constants $\delta,x,\veps>0$
\beqn
\label{lde4}
\limsup_{n\to\infty}\, n\,P\bigl(C_{x,\delta,\veps}^{(n,1)}\bigr)\leq
\lim_{n\to\infty} n\,P\bigl(\phi_1\circ X_0> \veps xb_n,Z_1>\veps xb_n\bigr)=0.
\feqn
Furthermore,
\beqn
\label{lde}
&&
P\bigl(D_{x,\delta,\veps}^{(n,1)}\bigr) \leq P\bigl(\phi_1\circ X_0>(1-\veps) xb_n,X_0\leq x(1+\delta)b_n\bigr)
\nonumber
\\
&&
\qquad
\leq
P\bigl(\phi_1\circ X_0>(1-\veps) xb_n\bigr| X_0\leq x(1+\delta)b_n\bigr)
\leq
P\Bigl(\sum_{i=1}^{\lfloor x(1+\delta)b_n\rfloor} B_{0,i}>(1-\veps) xb_n\Bigr)
\nonumber
\\
&&
\qquad
=E\Bigl[P_\Phi\Bigl(\frac{1}{x(1+\delta)b_n}\sum_{i=1}^{\lfloor x(1+\delta)b_n\rfloor} B_{0,i}>\frac{1-\veps}{1+\delta}\Bigr)\Bigr].
\feqn
Assume now that the constants $\delta>0$ and $\veps>0$ are chosen so small that $\frac{1-\veps}{1+\delta}>E[\phi_0],$ and hence
\beqn
\label{delta}
\frac{1-\veps}{1+\delta}>\eta E[\phi_0]\quad \, \mbox{for some}\quad\,\eta>1.
\feqn
We next derive a simple large-deviations type upper bound for the right-most expression in \eqref{lde}.
Denote $x_0=\frac{1-\veps}{1+\delta}.$ It follows from Chebyshev's inequality that for any $\lambda>0,$
\beq
E\Bigl[P_\Phi\Bigl(\frac{1}{n}\sum_{i=1}^n B_{0,i}>\frac{1-\veps}{1+\delta}\Bigr)\Bigr]
\leq e^{-n\lambda x_0} E\bigl[(1-\phi_0+\phi_0e^\lambda)^n\bigr].
\feq
Thus for all $\lambda>0$ small enough, namely for all $\lambda>0$ such that $e^\lambda<1+\eta\lambda,$ we have
\beq
&&
E\Bigl[P_\Phi\Bigl(\frac{1}{n}\sum_{i=1}^n B_{0,i}>\frac{1-\veps}{1+\delta}\Bigr)\Bigr]
\leq e^{-n\lambda x_0} E\bigl[\bigl(1-\phi_0+\phi_0(1+\eta\lambda) \bigr)^n\bigr]
\\
&&
\qquad
=e^{-n\lambda x_0} E\bigl[(1+\phi_0\eta\lambda)^n\bigr]\leq e^{-n\lambda x_0} E\bigl[ e^{\phi_0\cdot n\eta\lambda}\bigr].
\feq
Therefore, for all $\lambda>0$ small enough we have
\beq
\limsup_{n\to\infty} \frac{1}{n}\log E\Bigl[P_\Phi\Bigl(\frac{1}{n}\sum_{i=1}^n B_{0,i}>\frac{1-\veps}{1+\delta}\Bigr)\Bigr]
\leq -\lambda x_0+\log E\bigl[ e^{\eta\lambda\phi_0}\bigr].
\feq
Given $\eta,$ let $f(\lambda)= \log E\bigl[ e^{\eta\lambda\phi_0}\bigr].$ By the bounded convergence theorem,
$f'(0)=\eta E[\phi_0].$ Hence, in view of \eqref{delta},
\beq
\limsup_{n\to\infty} \frac{1}{n}\log P\Bigl(\frac{1}{n}\sum_{i=1}^n B_{0,i}>\frac{1-\veps}{1+\delta}\Bigr)<0.
\feq
Since $b_n$ is a regularly varying sequence, it follows from \eqref{lde} that
\beqn
\label{lde1}
\lim_{n\to\infty}\, n\,P\bigl(D_{x,\delta,\veps}^{(n,1)}\bigr)=0.
\feqn
Therefore, since $\veps>0$ above can be made arbitrary small (in particular, the left-hand side of \eqref{delta}
is an increasing function of $\veps$), combining \eqref{lde1} together with \eqref{lde4} and \eqref{lde} yields:
\beq
\limsup_{n\to\infty} P\bigl(A_{x,\delta}^{(n)}\bigr) \leq P\bigl(x b_n<K_n\leq x(1+\delta)b_n\bigr),
\feq
and hence
\beq
&&
\limsup_{n\to\infty} P(M_n> xb_n) = \limsup_{n\to\infty} \sum_{k=0}^\infty
P\bigl((1+k\delta)xb_n<M_n\leq (1+k\delta+\delta)xb_n\bigr)
\\
&&
\qquad  \leq \sum_{k=0}^\infty
\limsup_{n\to\infty} P\bigl((1+k\delta)xb_n<M_n\leq (1+k\delta+\delta)xb_n\bigr)
\\
&&
\qquad  \leq \sum_{k=0}^\infty
P\bigl((1+k\delta)xb_n<K_n\leq (1+k\delta+\delta)xb_n\bigr)=P(K_n>xb_n).
\feq
The proof of Theorem~\ref{extremes} is complete.\qed
\subsection{Proof of Theorem~\ref{partial}}
\label{partial-proof}
For $n\in\zz,$ let
\beqn
\label{yn}
Y_n=\sum_{t=n}^\infty X_{n,t}
\feqn
be the total number of progeny at all generations of all the immigrants entered at time $n,$
including the immigrants themselves. Then
\beq
\sum_{k=1}^n X_k&=&\sum_{k=1}^n \sum_{t=0}^k X_{t,k}=\sum_{t=0}^n \sum_{k=t}^n X_{t,k}=\sum_{t=0}^n \Bigl(\sum_{k=t}^\infty X_{t,k}
-\sum_{k=n+1}^\infty X_{t,k}\Bigr)
\\
&=&
\sum_{t=0}^n Y_t-\sum_{t=0}^n\sum_{k=n+1}^\infty X_{t,k}.
\feq
Notice that
\beq
\sum_{t=0}^n\sum_{k=n+1}^\infty X_{t,k}&=_D& \sum_{t=-n}^0\sum_{k=1}^\infty X_{t,k}
\leq \sum_{t=-\infty}^0\sum_{k=1}^\infty X_{t,k}
\\
&=&\sum_{t=-\nu_{-1}}^0\sum_{k=1}^\infty X_{t,k}\leq \sum_{t=-\nu_{-1}}^{\nu_1}Y_t <\infty.
\feq
Hence, in order to show that $S_n/b_n$ converges in distribution, it suffices to show that  $b_n^{-1}\sum_{k=1}^n Y_k$ converges to the same limit.
Note that the sequence $(Y_n)_{n\in\zz}$ has the same distribution under $P_0$ as it has under $P.$
\par
The following series of technical lemmas will enable us to
apply a general stable limit theorem (namely, Theorem~1.1 in \cite{fltg}; see also Corollary~5.7 in \cite{kobus})
to the partial sums of the sequence $Y_n.$
\begin{lemma}
\label{parta}
The sequence $(Y_n)_{n\in\zz}$ is strongly mixing. That is,
$\lim_{n\to\infty} \chi(n)=0,$ where
\beq
\chi(n):=\sup\bigl\{P(A \cap B)-P(A)P(B): A\in \calf^n,B\in
\calf_0\bigr\},
\feq and $\calf^n:=\sigma(Y_i: i\geq n),$ $\calf_n:=\sigma(Y_i: i < n).$
\end{lemma}
\begin{proof}[Proof of Lemma~\ref{parta}]
This is a variation of Lemma~3.2 in \cite{multi}.
For the sake of completeness we give here a suitable modification of the argument.
For $n\in\zz,$ let $\caly_n$ and $\caly^n$ denote, respectively, the sequences $(Y_i)_{i < n}$ and $(Y_i)_{i \geq n}.$
On one, for any $A \in \sigma(Y_i: i >n)$ and $B \in \sigma(Y_i: \leq 0),$ \beq &&
P\bigl(\caly^n \in A, \caly_0 \in B\bigr) \geq P\bigl(\caly^n \in A, \caly_0 \in B, \nu_1 \leq n/2\bigr)\\
&& \qquad =E\bigl[P_\Phi\bigl(\caly_0 \in B, \nu_1 \leq n/2\bigr)\cdot P_\Phi\bigl(\caly^n \in A\bigr)\bigr] \\
&& \qquad \geq P\bigl(\caly_0 \in B, \nu_1 \leq n/2\bigr)\cdot P\bigl(\caly^n \in A\bigr) \\&& \qquad
\geq
P\bigl(\caly_0 \in B\bigr)\cdot P\bigl(\caly^n \in A\bigr)-P(\nu_1 >n/2).
\feq
On the other hand,
\beq
&& P\bigl(\caly^n \in A, \caly_0 \in B\bigr) \leq P\bigl(\caly^n \in A, \caly_0 \in B, \nu_1 \leq n/2\bigr) +P(\nu_1 >n/2)
\\ && \qquad =E\bigl[P_\Phi\bigl(\caly_0 \in B, \nu_1 \leq n/2\bigr)\cdot P_\Phi\bigl(\caly^n \in A\bigr)\bigr]+
P(\nu_1>n/2)
\\&& \qquad \leq P\bigl(\caly_0 \in B, \nu_1 \leq n/2\bigr)\cdot P\bigl(\caly^n \in A\bigr)+
P(\nu_1>n/2)\\&& \qquad \leq P\bigl(\caly_0 \in B\bigr)\cdot P\bigl(\caly^n \in A\bigr)+
P(\nu_1>n/2).
\feq
It thus remains to show that $P(\nu_1<\infty)=1.$ By Proposition~\ref{rect}, we have
$P_0(\nu_1<\infty)=1.$ Since, clearly, $P(\phi_1\circ X_0=0)>0,$ the strong Markov property implies
$P(\nu_1<\infty)>0.$ Since the Markov chain $(X_n,Z_n)$ forms an ergodic process
according to Corollary~\ref{corol}, it follows from the ergodic theorem that the
the two-component Markov chain spends asymptotically a positive proportion of time at the set $\{X_n=Z_n\}$
(one can also appeal directly to the Poincar\'{e} recurrence theorem). This completes
the proof of the lemma.
\end{proof}
In view of the previous lemma we are seeking to apply to $Y_n$ the following general limit theorem for strongly mixing
stationary sequences obtained in \cite{fltg} (see also a similar Corollary~5.7 in \cite{kobus}).
\begin{theorem}
\cite[Theorem~1.1 and Corollary~1.4]{fltg}
\label{stabs}
Let $(Y_n)_{n\in\nn}$ be a stationary strongly mixing sequence of non-negative random variables.
Assume that for some $\alpha\in(0,1),$ there exists $\,h\in \calr_\alpha$ such that $\lim_{t\to\infty} h(t) \cdot P(Y_n>t)=1.$
For $n\in\nn,$ define a process $U_n$ on the Skorokhod space $D(\rr_+,\rr)$ by setting
\beq
U_n(t) =\frac{1}{b_n}\sum_{k=1}^{\lfloor nt\rfloor}Y_k,\qquad t\geq 0,
\feq
where $b_n$ are defined in \eqref{bn}. Then $U_n$ converges weakly in $D(\rr_+,\rr),$ as $n\to\infty,$ to a L\'{e}vy $\alpha$-stable
process if and only if the following local dependence condition holds:
\beqn
\label{ldep}
\mbox{\rm For any}~\veps > 0,~ \mbox{\rm we have:}~\quad \lim_{k\to\infty}\limsup_{n\to\infty} \,n\sum_{j=2}^{\lfloor n/k \rfloor}
P\bigl(Y_j  > \veps b_n, Y_1 > \veps b_n\bigr) = 0.
\feqn
\end{theorem}
We remark that the assumption $P(Y_n\in\zz_+)=1$ is actually not needed and is not included in the original version of the above theorem,
as it is stated in \cite{fltg}. It is not hard to verify that in our setting the random variable $Y_1$ has regularly varying distribution tails under the law $P_\Phi.$
To transform this statement into a corresponding claim under $P$ we will need the following a-priori bound.
\begin{lemma}
\label{parta3}
Let Assumption~\ref{assume1} hold. Then
\beqn
\label{310}
\limsup_{x\to\infty}\, h(x)\cdot P(Y_1>x)=C<\infty,
\feqn
where $C\in(0,\infty)$ is a positive constant whose value depends on the distribution of $\phi_0$ but not on the distribution of $Z_0$
(as long as Assumption~\ref{assume1} holds and $h(x)$ is defined as in (A2)).
\end{lemma}
\begin{proof}[Proof of Lemma~\ref{parta3}]
For any $x>0$ and $\gamma\in(0,1),$
\beq
P(Y_1>x)=P\Bigl(\sum_{n=1}^\infty X_{1,n}> x(1-\gamma)\sum_{n=1}^\infty \gamma^{n-1}\Bigr)
\leq
\sum_{n=1}^\infty  P\bigl(X_{1,n}>x\gamma^{n-1} (1-\gamma)\bigr).
\feq
Therefore,
\beq
&&
\limsup_{x\to\infty} h(x)\cdot P(Y_1>x)
\leq \sum_{n=1}^\infty \limsup_{x\to\infty} h(x)\cdot P\bigl(X_{1,n}>x\gamma^{n-1} (1-\gamma)\bigr)
\\
&&
\qquad
=
\sum_{n=1}^\infty \limsup_{x\to\infty} \frac{h(x)}{h\bigl(x\gamma^{n-1} (1-\gamma)\bigr)} \cdot h\bigl(x\gamma^{n-1} (1-\gamma)\bigr)
\cdot P\bigl(X_{1,n}>x\gamma^{n-1} (1-\gamma)\bigr)
\\
&&
\qquad
=
\sum_{n=1}^\infty \gamma^{-\alpha(n-1)}(1-\gamma)^{-\alpha} \cdot \limsup_{x\to\infty}  h(x)\cdot P(X_{1,n}>x).
\feq
Applying Lemma~\ref{klemma} to the right-most expression in this inequality, we obtain
by virtue of the bounded convergence theorem that
\beq
&&\limsup_{x\to\infty} h(x)\cdot P(Y_1>x)
\leq \sum_{n=1}^\infty \gamma^{-\alpha(n-1)}(1-\gamma)^{-\alpha} \cdot E\Bigl[\lim_{x\to\infty} h(x)\cdot P_\Phi(X_{1,n}>x)\Bigr]
\\
&&
\qquad
=\sum_{n=1}^\infty \bigl(\gamma^{-\alpha} \cdot E[\phi_0^\alpha]\bigr)^{n-1} (1-\gamma)^{-\alpha}.
\feq
Choosing now $\gamma\in(0,1)$ such that $\gamma>E[\phi_0^\alpha]$ concludes the proof of the lemma. To justify the above application
of the bounded convergence theorem, observe that $X_{1,n}\leq Z_1$ and $Z_1$ is independent of $\Phi.$
\end{proof}
In order to study the exact asymptotic of the distribution tails of $Y_1,$ it is convenient to approximate $Y_1$ by $Y_1^{(m)},$ where
\beq
Y_n^{(m)}:=\sum_{k=n}^{n+m} X_{n,k},\qquad n\in\zz.
\feq
We have:
\begin{lemma}
\label{parta4}
Let Assumption~\ref{assume1} hold. Then
\beqn
\label{soon}
\lim_{x\to\infty} h(x)\cdot P(Y_1^{(m)}>x)=E\Bigl[\Bigr(1+\sum_{i=1}^m \prod_{j=1}^i  \phi_j\Bigl)^\alpha\Bigr],
\feqn
for any $m\in\nn.$
\end{lemma}
\begin{proof}[Proof of Lemma~\ref{parta4}]
Note that
\beq
Y_n^{(m)}=\sum_{k=1}^{Z_n}\Bigl(1+\sum_{i=1}^m B_{n,k}^{(i)}\Bigr),
\feq
where $B_{n,k}^{(i)}$ is the number of progeny (either zero or one) of the $k$-th immigrant at generation $n,$
who is present (or not) at the system at generation $n+i.$ Then an
argument similar to the one which we have employed in order to prove Lemma~\ref{klemma} (see also Remark~\ref{remark1})
along with \eqref{sumrv} ensure that
\beq
\lim_{x\to\infty} h(x)\cdot P_\Phi\bigl(Y_0^{(m)}>x\bigr)=\Bigl(1+\sum_{i=1}^m E_\Phi\bigl[B_{0,1}^{(i)}\bigr]\Bigr)^\alpha
=\Bigr(1+\sum_{i=1}^m \prod_{j=1}^i  \phi_j\Bigl)^\alpha.
\feq
Since $Y_0^{(m)}\leq mZ_0$ and $Z_0$ is independent of $\Phi,$ the bounded convergence theorem yields
\beq
\lim_{x\to\infty} h(x)\cdot P\bigl(Y_0^{(m)}>x\bigr)=E\bigl[\lim_{x\to\infty} h(x)\cdot P_\Phi\bigl(Y_0^{(m)}>x\bigr)\bigr]
=
E\Bigl[\Bigr(1+\sum_{i=1}^m \prod_{j=1}^i  \phi_j\Bigl)^\alpha\Bigr],
\feq
completing the proof of the lemma.
\end{proof}
Combining together the results of Lemmas~\ref{parta3} and~\ref{parta4} we can deduce the following:
\begin{lemma}
\label{parta1}
Let Assumption~\ref{assume1} hold. Then
\beq
\lim_{x\to\infty} h(x)\cdot P(Y_1>x)=E\Bigl[\Bigr(1+\sum_{i=1}^\infty\prod_{j=1}^i  \phi_j\Bigl)^\alpha\Bigr]<\infty.
\feq
\end{lemma}
\begin{proof}[Proof of Lemma~\ref{parta1}]
First, observe that the lower bound
\beq
\lim_{x\to\infty} h(x)\cdot P(Y_1>x)\geq \lim_{m\to\infty} \lim_{x\to\infty} h(x)\cdot P\bigl(Y_1^{(m)}>x\bigr)
=E\Bigl[\Bigr(1+\sum_{i=1}^\infty\prod_{j=1}^i  \phi_j\Bigl)^\alpha\Bigr]
\feq
holds by virtue of Lemma~\ref{parta4} and the monotone convergence theorem.
\par
To prove the matching upper bound, notice that the difference $Y_1-Y_1^{(m)}$ is distributed under the law $P$ as $Y_1$
is distributed under the law $Q,$ where $Q$ is defined in the same way as $P$ with the only exception that in the former case the
distribution of $Z_n$ is assumed to be that of $\Pi_{m+1}\circ Z_0$ under $P.$ Furthermore, since $\Pi_{m+1}\circ Z_0\leq Z_0,$
Lemma~\ref{klemma} and the bounded convergence theorem imply that
\beq
\lim_{x\to\infty} h(x)\cdot P\bigl(\Pi_{m+1}\circ Z_0>x\bigr)=\bigl(E[\phi_0^\alpha]\bigr)^{m+1}.
\feq
It follows then from \eqref{310} with the probability measure $P$ replaced by $Q,$ that
\beq
\lim_{m\to\infty} \limsup_{x\to\infty} h(x)\cdot P\bigl(Y_1-Y_1^{(m)}>x\bigr)=0.
\feq
Thus, using again Lemma~\ref{parta4} and the monotone convergence theorem, we obtain that the following
holds for any $\veps>0:$
\beq
&&\limsup_{x\to\infty} h(x)\cdot P(Y_1>x)
\\
&&
\qquad
\leq \lim_{m\to\infty} \Bigl\{\lim_{x\to\infty} h(x)\cdot P\bigl(Y_1^{(m)}>x(1-\veps)\bigr)+
\limsup_{x\to\infty} h(x)\cdot P\bigl(Y_1-Y_1^{(m)}>x\veps\bigr)\Bigr\}
\\
&&
\qquad
=\lim_{m\to\infty} \lim_{x\to\infty} h(x)\cdot P\bigl(Y_1^{(m)}>x(1-\veps)\bigr)
=E\Bigl[\Bigr(1+\sum_{i=1}^\infty\prod_{j=1}^i  \phi_j\Bigl)^\alpha\Bigr]\cdot (1-\veps)^{-\alpha}.
\feq
Taking $\veps\to 0$ yields the desired upper bound. To conclude the proof of the lemma it remains to note
that by Jensen's inequality,
\beq
E\Bigl[\Bigr(1+\sum_{i=1}^\infty\prod_{j=1}^i  \phi_j\Bigl)^\alpha\Bigr]\leq
\Bigr(E\Bigl[1+\sum_{i=1}^\infty\prod_{j=1}^i  \phi_j\Bigr]\Bigl)^\alpha=
\bigr(1-E[\phi_0]\bigr)^{-\alpha}<\infty,
\feq
where we used the assumption $\alpha\in(0,1).$
\end{proof}
We are now in a position to complete the proof of Theorem~\ref{partial}.
It suffices to verify that the conditions of Theorem~\ref{stabs} hold for the sequence $(Y_n)_{n\geq 1}.$
In view of Lemmas~\ref{parta} and~\ref{parta1}, we only need to check the validity of the ``local dependence" condition
\eqref{ldep}. To this end, observe that for any $j\geq 2,$ $Y_j$ and $Y_1$ are independent under the law $P_\phi,$
and hence Cauchy-Schwarz inequality yields
\beq
&&
P\bigl(Y_j  > \veps b_n, Y_1 > \veps b_n\bigr) =
E\bigl[P_\Phi(Y_j > \veps b_n)\cdot P_\Phi(Y_1  > \veps b_n) \bigr]
\leq E\bigl[P_\Phi^2(Y_1 > \veps b_n)\bigr] .
\feq
An argument similar to the one we employed to prove Lemma~\ref{parta1} shows then that the
following limit exists and the identity holds:
\beq
\lim_{x\to\infty} h(x)^2\cdot  E\bigl[P_\Phi^2(Y_1 > x)\bigr]=E\Bigl[\Bigr(1+\sum_{i=1}^\infty\prod_{j=1}^i  \phi_j\Bigl)^{2\alpha}\Bigr]<\infty.
\feq
Thus
\beq
\lim_{n\to\infty} n^2\cdot  E\bigl[P_\Phi^2(Y_1 > \veps b_n)\bigr]=\veps^{-2\alpha}\cdot E\Bigl[\Bigr(1+\sum_{i=1}^\infty\prod_{j=1}^i  \phi_j\Bigl)^{2\alpha}\Bigr]<\infty,
\feq
and
\beq
&&
\lim_{k\to\infty}\limsup_{n\to\infty} \,n\sum_{j=2}^{\lfloor n/k \rfloor} P\bigl(Y_j  > \veps b_n, Y_1 > \veps b_n\bigr)
\leq
\lim_{k\to\infty}\limsup_{n\to\infty} \frac{n^2}{k}\cdot  E\bigl[P_\Phi^2(Y_1 > \veps b_n)\bigr]
\\
&&
\qquad
=
\lim_{k\to\infty}\limsup_{n\to\infty} \frac{n^2}{k}\cdot n^{-2}\veps^{-2\alpha}\cdot E\Bigl[\Bigr(1+\sum_{i=1}^\infty\prod_{j=1}^i  \phi_j\Bigl)^{2\alpha}\Bigr] =0,
\feq
as desired. The proof of Theorem~\ref{partial} is completed. \qed
\subsection{Proof of Lemma~\ref{tails}}
\label{tails-proof}
Recall $Y_n$ from \eqref{yn}. Define
\beq
Q_n=X_n+\mbox{total progeny of the $X_n$ particles present at generation}~n.
\feq
For all $A>0$ define its stopping time $\varsigma_A=\inf\{n: X_n>A\}.$
The random variable $W_1$ can be represented on the event
$\{\varsigma_A< \nu_1\}$ in the following form:
\beqn
\label{piruk-matrif}
W_1=\sum_{n=0}^{\varsigma_A-1}X_n+ Q_{\varsigma_A}+\sum_{\varsigma_A <n <  \nu_1} Y_n.
\feqn
The three terms in the right-hand side
of \eqref{piruk-matrif} are evaluated in the following series of
lemmas. It will turn out that for large $A,$ the main contribution to $ W_1$ in
\eqref{piruk-matrif} comes from the second term.
Fix any $\delta>0.$ It follows from \eqref{estimate} that for any $A>0,$
\beq
P_0\Bigl(\sum_{n=0}^{\min\{\varsigma_A, \nu_1\}-1} X_n \geq \delta t  \Bigr)
\leq P_0(A \nu_1 \geq \delta t)\leq K_1 e^{-K_2 \delta t/A},
\feq
and hence
\beqn
\label{dereh1}
P_0( W_1 \geq \delta t,\varsigma_A \geq \nu_1) \leq P(A \nu_1 \geq \delta t)\leq K_1 e^{-K_2 \delta t/A},
\feqn
\beqn
\label{dereh2}
P_0\Bigl(\sum_{n=0}^{\varsigma_A-1}X_n \geq \delta t,\varsigma_A <  \nu_1\Bigr)
\leq P_0(A \nu_1 \geq \delta t) \leq K_1 e^{-K_2\delta t/A}.
\feqn
\begin{lemma}
\label{vai}
For all $\delta>0$ there exists an
$A_0=A_0(\delta)<\infty$ such that
\beqn
\label{cor17}
h(t) \cdot P_0\Bigl(\sum_{\varsigma_A < n < \nu_1 }Y_n \geq \delta t\Bigr) \leq \delta, \qquad \mbox{\rm for all}~A \geq A_0~\mbox{\rm and}~ t>0.
\feqn
\end{lemma}
\begin{proof}[Proof of Lemma~\ref{vai}]
\label{proof-vai}
Using the identity $\sum_{n=1}^\infty n^{-2}=\pi^2/6 <2$ and the fact that $Y_n$ is independent of
$\one{\varsigma_A < n < \nu_1}$ under the law $P_0,$ we obtain that the following holds for all $t>0:$
\beqn
\label{almost}
&&
h(t) \cdot P_0\Bigl(\sum_{\varsigma_A < n < \nu_1}Y_n \geq \delta t \Bigr)
=
h(t) \cdot P_0\Bigl(\sum\limits_{n=1}^\infty Y_n \one{\varsigma_A < n < \nu_1}
\geq 6\delta t \pi^{-2} \sum\limits_{n=1}^\infty n^{-2}\Bigr)
\\
\nonumber
&&
\qquad
\leq \sum\limits_{n=1}^\infty P_0(\varsigma_A < n<\nu_1)\cdot h(t) \cdot P_0\bigl(Y_n \geq 1/2 \cdot \delta t n^{-2}\bigr)
\\
\nonumber
&&
\qquad
\leq \sum\limits_{n=1}^\infty P_0(\varsigma_A < n<\nu_1)\cdot\frac{ h(t)}{h( 1/2 \cdot \delta t n^{-2})} \cdot h( 1/2 \cdot \delta t n^{-2}) \cdot P_0\bigl(Y_n \geq 1/2 \cdot \delta t n^{-2}\bigr)
\feqn
To bound the term $\frac{ h(t)}{h( 1/2 \cdot \delta t n^{-2})},$ we apply the following simplified version of \cite[Lemma~1]{grey}:
\beqn
\label{potterlike}
\mbox{There exists}~K>1~\mbox{such that}~\frac{h(\lambda t)}{h(t)}\leq K(\lambda^{-\alpha}+\lambda^{\alpha})~\mbox{for all}~\lambda>0,\,t>0.
\feqn
It follows that
\beq
\frac{ h(t)}{h( 1/2 \cdot \delta t n^{-2})}\leq K 2^\alpha n^{2\alpha} (\delta^{-\alpha}+\delta^\alpha),\qquad t>0.
\feq
Using Lemma~\ref{parta1} and \eqref{potterlike}, we obtain from \eqref{almost} and the above bound
that the following holds for all $t>0$ with a suitable constant $C>0$ independent of $n,\delta,A$ and $t:$
\beq
&&
h(t) \cdot P_0\Bigl(\sum_{\varsigma_A < n < \nu_1}Y_n \geq \delta t \Bigr)
\leq C 2^\alpha t^{-\alpha}(\delta^{-\alpha}+\delta^\alpha)E_0\bigl[ \nu_1^{2\alpha+1}; \varsigma_A< \nu_1\bigr]
\\
&&
\qquad
\leq C 2^\alpha t^{-\alpha}(\delta^{-\alpha}+\delta^\alpha)\sqrt{E_0\bigl(\nu_1^{4\alpha+2}\bigr)}\cdot \sqrt{P_0(\varsigma_A < \nu_1)}.
\feq
The claim follows now from \eqref{estimate}, the first square root being bounded and the second one going to zero as $A
\to \infty.$
\end{proof}
It follows from \eqref{piruk-matrif}, taking estimates
\eqref{dereh1}, \eqref{dereh2} and \eqref{cor17} into account,
that for any $A>A_0(\delta)$ (where $A_0$ is given by
\eqref{cor17}) there exists $t_A>0$ such that
\beqn
\label{approx-w}
&&
h(t)\cdot P_0(\varsigma_A < \nu_1,Q_{\varsigma_A} \geq t)
\leq
h(t)\cdot P_0( W_1 \geq t)
\nonumber
\\
&&
\qquad
\leq
h(t)\cdot P_0(\varsigma_A<\nu_1,Q_{\varsigma_A} \geq t(1-2\delta))+3\delta,
\feqn
for all $t>t_A.$ Thus, $ W_1$ can be approximated by $Q_{\varsigma_A}.$
The following lemma deals with the distribution tails of the latter.
\begin{lemma}
\label{parta14}
Let Assumption~\ref{assume1} hold. Then:
\item [(a)] We have:
\beqn
\label{partials}
\limsup_{A\to\infty}\, \limsup_{t\to\infty}
h(t)\cdot P_0 \bigl(X_{\varsigma_A}\geq t, \varsigma_A < \nu_1\bigr) <\infty.
\feqn
\item [(b)] The following limit exists and is finite for any given $A>0:$
\beq
\lim_{t\to\infty} h(t)\cdot P_0 \bigl(X_{\varsigma_A}\geq t, \varsigma_A < \nu_1\bigr).
\feq
\item [(c)] The following limit exists and is finite for any given $A>0:$
\beq
\lim_{t\to\infty} h(t)\cdot P_0 \bigl(Q_{\varsigma_A}\geq t, \varsigma_A < \nu_1\bigr).
\feq
\end{lemma}
\begin{proof}[Proof of Lemma~\ref{parta14}]
$\mbox{}$
\\
{\bf (a)} Recall $M_n$ from \eqref{mn}. For $t>A$ we have:
\beq
&&
P_0 \bigl( X_{\varsigma_A}>t;\varsigma_A<\nu_1\bigr)
=
\sum_{n \geq 1} \sum_{a=0}^{A-1} P_0 \bigl( X_{\varsigma_A}>t,
\varsigma_A =n, X_n>A, X_{n-1}=a,\nu_1>n\bigr)
\\
&&
\qquad
=
\sum_{n \geq 1} \sum_{a=0}^{A-1} P_0 \bigl( X_n>t, M_{n-1}<A, X_{n-1}=a,\nu_1>n\bigr)
\\
&&
\qquad
\leq
\sum_{n \geq 1}
P \bigl( Z_n>t-A,M_{n-1}<A, \nu_1>n\bigr)
\\
&&
\qquad
=
\sum_{n \geq 1}
P\bigl( Z_n>t-A) \cdot P(M_{n-1}<A, \nu_1>n\bigr)
\leq  P \bigl( Z_0>t-A) \cdot E_0[\nu_1].
\feq
In view of Assumption~\ref{assume1} and \eqref{estimate}, this completes the proof of part (a).
\\
$\mbox{}$
\\
{\bf (b)}  The computation is quite similar to the one in part (a).
Namely, for $t>A$ we have:
\beq
&&
P_0 \bigl( X_{\varsigma_A}>t;\varsigma_A<\nu_1\bigr)
=
\sum_{n \geq 1} \sum_{a=1}^{A-1} P_0 \bigl( Z_n>t-a, \phi_n \circ X_{n-1}=a, M_{n-1}<A,\nu_1>n\bigr)
\\
&&
\qquad
=\sum_{n \geq 1} \sum_{a=1}^{A-1} P(Z_0>t-a)\cdot P_0 \bigl(\phi_n \circ X_{n-1}=a, M_{n-1}<A,\nu_1>n\bigr).
\\
&&
\qquad
=\sum_{a=1}^{A-1} P(Z_0>t-a)\cdot \sum_{n \geq 1} P_0 \bigl(\phi_n \circ X_{n-1}=a, M_{n-1}<A,\nu_1>n\bigr)
\feq
As before,
\beq
\sum_{n \geq 1} P_0\bigl(\phi_n \circ X_{n-1}=a,M_{n-1}<A, \nu>n\bigr)\leq \sum_{n \geq 1} P_0(\nu_1>n)=E[\nu_1]<\infty,
\feq
from which the claim of part (b) follows in view of Assumption~\ref{assume1}.
\\
$\mbox{}$
\\
{\bf (c)} This is merely Lemma~\ref{parta1} applied to $X_{\varsigma_A}$ under the conditional law $P(\,\cdot\,|\,\varsigma_A<\nu_1)$ rather
than to $Z_1$ under the regular measure $P.$
\end{proof}
We are now in a position to conclude the proof of Lemma~\ref{tails}.
It follows from \eqref{approx-w}, \eqref{potterlike}, and Lemma~\ref{parta14} that
\beq
\lim_{t \to \infty} h(t)\cdot P_0( W_1 > t)
=\lim_{A \to \infty} \lim_{t\to\infty} h(t)\cdot P_0\bigr(Q_{\varsigma_A} > t;\varsigma_A<\nu_1\bigr)<\infty.
\feq
The second limit, taken as $A\to\infty,$ in the right-hand side exists since the limit in the left-hand side does not depend of $A.$
Furthermore, by Assumption~\ref{assume1},
\beq
\lim_{t \to \infty} h(t)\cdot P_0( W_1 > t)\geq \lim_{t \to \infty} h(t)\cdot P( Z_1 > t)>0,
\feq
concluding the proof of Lemma~\ref{tails}. \qed
\section*{Appendix: Proof of two auxiliary propositions}
\subsection*{Proof of Proposition~\ref{convergence}}
{\bf (a)} By Jensen's inequality, if $E[Z_0^\beta]<\infty$ for $\beta>0,$ then 
$E[Z_0^{\beta/m}]<\infty$ for any $m\in\nn.$ Therefore, without loss of generality we can assume that $\beta \in (0,1)$ in Assumption~\ref{assume4}.
Assuming from now on and throughout the proof of part (a) of Proposition~\ref{convergence} that $\beta \in (0,1),$ we obtain 
by virtue of Jensen's inequality for conditional expectations that
\beqn
\label{estim}
&&
E_0\bigl[(\Pi_k\circ Z_k)^\beta\bigr]=E_0\bigl[E_0\bigl[(\Pi_k\circ Z_k)^\beta|\Phi,\calz\bigr]\bigr]
\leq E_0\bigl[\bigl(E_0\bigl[\Pi_k\circ Z_k|\Phi,\calz\bigr]\bigr)^\beta\bigr]
\nonumber
\\
&&
\qquad
=E\Bigl[\Bigl(\prod_{j=1}^k \phi_j \cdot Z_k\Bigr)^\beta\Bigr]=
E[Z_0^\beta]\cdot \bigl(E[\phi_0^\beta]\bigr)^k.
\feqn
Hence
\beq
E[X_\infty^\beta]=E\Bigl[\Bigl(\sum_{k=0}^\infty X_{0,k}\Bigr)^\beta\Bigr]\leq \sum_{k=0}^\infty E\bigl[X_{0,k}^\beta\bigr]
\leq E[Z_0^\beta]\cdot \sum_{k=0}^\infty  \bigl(E[\phi_0^\beta]\bigr)^k<\infty.
\feq
In particular, $X_\infty$ is $P-\as$ finite.
\\
$\mbox{}$
\\
{\bf (b)} For $n\in \nn,$ we have \beq
X_n=\sum_{k=1}^n X_{k,n}+X^{(0,n)},
\feq
where $X^{(0,n)}=_P\Pi_n \circ X_0.$ Since $P\bigl(\lim_{n\to\infty} \Pi_n\circ X_0=0\bigr)=1$ for any $X_0\in\caln_+,$ the limiting
distribution of $X_n,$ if exists, is independent of $X_0.$ Furthermore, if $X_0=0,$ the i.i.d. structure of $(\Phi,\calz)$ yields:
\beq
X_n=_P \sum_{k=-n+1}^0 X_{k,0}=_P Z_0+\sum_{k=1}^{n-1} \Pi_k \circ Z_k,
\feq
The claim of part (b) follows now from the almost sure convergence of the series on the right-hand side of the above identity to $X_\infty.$
\\
$\mbox{}$
\\
{\bf (c)}
To see that the stationary distribution is unique, consider two stationary solutions $\bigr(X_n^{(1)}\bigr)_{n\in\zz_+}$ and
$\bigr(X_n^{(2)}\bigr)_{n\in\zz_+}$ to \eqref{version}
corresponding to different initial values, $X_n^{(1)}$ and $X_n^{(2)},$ respectively. Then, since
$\Pi_n$ are ``thinning" operators,
\beq
|X_n^{(1)}-X_n^{(2)}|\leq \Pi_{n+1}\circ |X_0^{(1)}-X_0^{(2)}|,
\feq
and hence
\beq
\lim_{n\to\infty} \bigl(X_n^{(1)}-X_n^{(2)}\bigr)=0,\qquad P-\as
\feq
The proof of the proposition is complete. \qed
\subsection{Proof of Proposition~\ref{rect}}
{\bf (a)} By Corollary~\ref{corol}, $P_0\bigl(X_n=k_{\min}~\io\bigr)=1,$ and hence $P_0(\nu_n<\infty)=1$ for all $n\in\nn.$
The argument showing that the pairs $(\sigma_n,W_n)_{n\in\nn}$ form an i.i.d. sequence is standard (cf. \cite{renewal-appr})
and is based on the following two observation along with the use of the strong Markov property:
\begin{itemize}
\item [(i)] The random times $\nu_n$ are times of the successive visits to the set $\{(x,y)\in\zz^2:x=y\}$ by the two-component Markov chain
$(X_n,Z_n)_{n\in\nn}.$ Furthermore, $X_{\nu_n}=Z_{\nu_n}=_P Z_0.$
\item [(ii)] Transition kernel of the Markov chain $(X_n,Z_n)$ depends only on the current value of the first component,
but not on the value of the second.
\end{itemize}
$\mbox{}$
\\
{\bf (b)} In order to prove part (b) of the proposition, it suffices to show that the
following power series has a radius of convergence greater than $1:$
\beq
V(z)= \sum_{t=0}^\infty P_0(\sigma_1> t)z^t.
\feq
Let us introduce some notation. Let $v(t)= P_0(\sigma_1 > t),$
\beq
h(r, t)= P \Bigl(X_{r, t} \neq 0,\, \sum_{j=r+1}^{t-1} X_{j, t} = 0 \Bigr)
\qquad \mbox{and} \qquad
g(r, t)= P \Bigl(\sum_{j=r}^{t-1} X_{j, t} = 0 \Bigr).
\feq
Then
\beqn
\label{bp22}
v(t) &=& P_0(\sigma_1 > t, X_t\neq 0)= P_0\Bigl(\sigma_1>t, \sum_{k=0}^{t-1} X_{k,t}\neq 0\Bigr)
\nonumber
\\
&=&
\sum_{k=0}^{t-2} P_0\Bigl(\sigma_1>t, X_{k,t}\neq 0, \sum_{j=k+1}^{t-1} X_{k,t}=0\Bigr)
+ P_0(\sigma_1 > t, X_{t-1,t}\neq 0).
\feqn
Using the i.i.d. structure of the sequence of random coefficients $(\Phi,Z),$
we obtain:
\beqn
\label{bp23}
&& P_0\Bigl(\sigma_1>t, X_{k,t}\neq 0, \sum_{j=k+1}^{t-1} X_{k,t}=0\Bigr)
=P_0\Bigl(\sigma_1>k, X_{k,t}\neq 0, \sum_{j=k+1}^{t-1} X_{k,t}=0\Bigr)
\nonumber
\\
&&
\qquad
=P_0(\sigma_1>k)P\Bigl( X_{k,t}\neq 0, \sum_{j=k+1}^{t-1} X_{k,t}=0\Bigr)= v(k)h(k, t)
\feqn
and
\beqn
\label{bp24}
g(k, t) = P(X_{t - k} = 0).
\feqn
Let $g(t) = P_0(X_t = 0).$ It follows from \eqref{bp24} that
\beqn
\label{bp25}
g(k, t) = g(t-k ).
\feqn
Next, let $h(t) =g(t - 1)-g(t).$ Since $h(k, t) + g(k, t) =g(k + 1, t),$ then \eqref{bp25} implies that
\beqn
\label{bp26}
h(k,t) =h(t-k).
\feqn
Substituting \eqref{bp23} into \eqref{bp22} and then using \eqref{bp26} gives
\beq
v(t) = \sum_{k=0}^{t-1}v(k)h(t - k).
\feq
In addition, we have $v(0) = 1,$ $h(k) > 0$ for all $k > 0,$ and
\beq
\sum_{k=1}^\infty h(k)= 1 -\lim_{t\to\infty}g(t) =  1 -\lim_{t\to\infty}  P_0(X_t = 0)  < 1,
\feq
where for the last inequality we used Theorem~3.3. Therefore, $\{v(k): k = 0, 1, 2,\ldots \}$ is a renewal sequence.
Therefore (see \cite[Section~XIII.3]{fbook1}), $V(z) = \bigl(1 - H(z)\bigr)^{-1},$ where
\beq
 H(z):= \sum_{t=0}^\infty h(t)z^t.
\feq
To conclude the proof of the proposition, observe that, using \eqref{estim} and Chebyshev's inequality,
\beq
&&
h(t) = h(0, t) < P(X_{0,t}\neq 0) \leq E[X_{0,t}]= E[Z_0^\beta]\cdot \bigl(E[\phi_0^\beta]\bigr)^k,
\feq
and hence the radius of convergence of $H(z)$ is greater than $1.$ \qed
\providecommand{\bysame}{\leavevmode\hbox to3em{\hrulefill}\thinspace}

\end{document}